\documentclass{amsart}
\numberwithin{equation}{section}
\usepackage{amscd}
\usepackage{amssymb}
\let\cal\mathcal
\def\Ascr{{\cal A}}
\def\Bscr{{\cal B}}
\def\Cscr{{\cal C}}

\def\Escr{{\cal E}}
\def\Fscr{{\cal F}}

\def\Iscr{{\cal I}}

\def\Lscr{{\cal L}}
\def\Mscr{{\cal M}}

\def\Oscr{{\cal O}}

\def\Qscr{{\cal Q}}

\def\Vscr{{\cal V}}
\def\Wscr{{\cal W}}

\let\blb\mathbb

\def\GG{{\blb G}}
\def \PP{{\blb P}}
\def \AA{{\blb A}}
\def \ZZ{{\blb Z}}

\def\Id{\operatorname{id}}
\def\pr{\mathop{\text{pr}}\nolimits}

\def\Der{\operatorname{Der}}

\def\mod{\operatorname{mod}}

\def\gr{\operatorname{gr}}

\def\coh{\mathop{\text{\upshape{coh}}}}

\def\gr{\operatorname {gr}}
\def\Spec{\operatorname {Spec}}

\def\GL{\operatorname {GL}}
\def\PGL{\operatorname {PGL}}

\def\Hom{\operatorname {Hom}}
\def\uHom{\operatorname {\mathcal{H}\mathit{om}}}
\def\End{\operatorname {End}}

\def\Sl{\operatorname {Sl}}

\def\End{\operatorname {End}}
\def\Gal{\operatorname {Gal}}

\def\rk{\operatorname {rk}}

\def\Pic{\operatorname {Pic}}

\def\r{\rightarrow}

\DeclareMathOperator{\HHom}{\mathcal{H}\mathit{om}}
\DeclareMathOperator{\HEnd}{\mathcal{E}\mathit{nd}}

\DeclareMathOperator{\ram}{ram}
\DeclareMathOperator{\Br}{Br}

\let\dirlim\injlim
\let\invlim\projlim

\newtheorem{lemma}{Lemma}[section]
\newtheorem{proposition}[lemma]{Proposition}
\newtheorem{theorem}[lemma]{Theorem}
\newtheorem{corollary}[lemma]{Corollary}

\theoremstyle{definition}

\newtheorem{example}[lemma]{Example}

{

}

\theoremstyle{remark}

\newtheorem{remark}[lemma]{Remark}

\newdimen\uboxsep \uboxsep=1ex
\def\uboxn#1{\vtop to 0pt{\hrule height 0pt depth 0pt\vskip\uboxsep
\hbox to 0pt{\hss #1\hss}\vss}}

\def\uboxs#1{\vbox to 0pt{\vss\hbox to 0pt{\hss #1\hss}
\vskip\uboxsep\hrule height 0pt depth 0pt}}

\def\index{\operatorname{index}}
\def\period{\operatorname{period}}
\def\Ram{\operatorname{Ram}}
\def\PGl{\operatorname{PGl}}
\def\Gl{\operatorname{Gl}}
\def\Az{\operatorname{Az}}
\def\Vect{\operatorname{Vect}}
\def\cl{\operatorname{cl}}
\def\uEnd{\operatorname {\mathcal{E}\mathit{nd}}}
\let\SL\Sl
\let\GL\Gl
\let\PGL\PGl
\def\Sch{\operatorname{Sch}}
\title[Notes on de Jong's theorem]{Notes on de Jong's period$=$index theorem for central simple
algebras over fields of transcendence degree two}
 \email{michel.vandenbergh@uhasselt.be} 
\author{Michel Van den Bergh}
\address{Departement WNI\\Universiteit Hasselt\\ Universitaire Campus\\ Building D\\ 3590 
Diepenbeek\\ Belgium} 
\begin{document}
\maketitle
\section{Introduction}
These are notes on de Jong's proof of the period$=$index theorem
over fields of transcendence degree two. They are actually about the
``simplified'' proof sketched by de Jong in the last section of his
paper. These notes were meant as support for my lectures at the
summer school ``Central Simple Algebras over Function Fields of
Surfaces'' at the Universit\"at Konstanz between August, 26 and
September, 1 2007 but I did not finish them in time.\footnote{The
  other lecturers were Philippe Gille, Andrew Kresch, Max Lieblich,
  Tam\'as Szamuely and Jan Van Geel.}

\medskip

No originality is intended here. I have  copied mostly from the following
sources \nocite{CT1,OPA,AdJ}
\begin{itemize}
\item M. Artin and A.~J. de~Jong, {\em Stable orders over surfaces}, manuscript 2004, 198 pages. 
\item J.-L. Colliot-Th{\'e}l{\`e}ne, {\em Alg\`ebres simples centrales
    sur les corps de fonctions de deux variables (d'apr\`es {A}. {J}.
    de {J}ong)}, Ast\'erisque (2006), no.~307, Exp. No. 949, ix,
  379--413, S\'eminaire Bourbaki. Vol.  2004/2005.
\item A.~J. de~Jong, {\em The period-index problem for the {B}rauer group of an
  algebraic surface}, Duke Math. J. {\bf 123} (2004), no.~1, 71--94.
\item M. Ojanguren and R. Parimala, {\em   Smooth finite splittings of
Azumaya algebras over surfaces.} http://www.mathematik.uni-bielefeld.de/lag/man/244.ps.gz.
\end{itemize}
I wish to thank Johan de Jong for answering some questions and for
sending me his manuscript with Mike Artin on stable orders.

\medskip

\noindent
\textbf{Disclaimer and apology } These notes are not very systematic.
In particular there are some inconsistencies of notation and the level
of detail is very uneven. Several parts are directly copied from the
references listed above.  Some parts consider subjects covered by other
lecturers. Since these notes were mostly written before the
summer school my presentation here is often different in content and
notation from the actual lectures during the summer school. I apologize
for this.

\medskip

\noindent
\textbf{Notations and conventions } Unless otherwise specified all cohomology
is etale cohomology, $k$ is an algebraically closed field and
 $n$ is a number which is prime to the characteristic. 

As over
fields we define the degree of an Azumaya algebra as the square root of its
rank (assuming it has constant rank). 
\section{De Jong's theorem}
Here is de Jong's celebrated result.
\begin{theorem} \cite{dJ} If $K$ is a field of transcendence degree
  $2$ and if $A$ is a central simple algebra over $K$ whose Brauer
  class has order prime to the characteristic then
\[
\index(D)=\period(D)
\]
\end{theorem}
We start with some observations and comments.
\begin{enumerate}
\item Max Lieblich \cite{Lieblich} has proved that the period-index
  problem over function fields can be lifted to characteristic zero.
  Hence the condition that the period must be prime to the
  characteristic can be removed. However in these notes we stick to de
  Jong's original setup.
\item It is easy to reduce the period-index problem to the case that
$K$ is finitely generated. 
\item If $K$ is finitely generated then by resolution of singularities
we may assume that $K$ is the function field of a smooth projective
connected surface $X$ (notation: $K=k(X)$). 
\item If $K=k(X)$ as above then one may first consider the case that
$D$ is in the image of $\Br(X)$. This is the so-called ``unramified case''. 
De Jong's proof deals first with the unramified case. The proof uses
a deformation argument which is very specific for surfaces. 
\item The ``ramified'' case in de Jong's proof is treated by an
  extremely ingenious reduction to the unramified case.
\item In \cite{Lieblich} Max Lieblich gives a new and completely
  different proof of de Jong's theorem which starts by (birationally)
  fibering the surface $X$ over $\PP^1$. He then rephrases the
  period-index problem as an existence problem for certain twisted
  vector bundles on curves over $k(t)$.  The corresponding moduli
  space is (geometrically) rationally connected (rational even) and in this
way 
  one can recover the period-index theorem from
(deep) results about the existence of rational points on
  rationally connected varieties. See \cite{dJS}.
\item Yet another proof is by Starr and de Jong \cite{dJ2}. Here one reduces the problem
to the existence of rational points in certain twisted Grassmanians over the function
field of a surface. 
\end{enumerate}
\section{Specialization}
De Jong's proof uses several times a specialization argument for the index. We
discuss this first.
\begin{theorem} \label{ref-3.1-0} Assume that $L/K$ is a field extension and assume
that there is discrete valuation ring in $L$ which contains $K$, whose
quotient field is $L$ and whose residue field is $K$. Let $D$
be a central division algebra over $K$. Then 
\[
\index(L\otimes_K D)=\index(D)
\]
\end{theorem}
\begin{proof} Let $m$ be the maximal ideal of $R$. Put $A=R\otimes_K D$
and filter $A$ by the $m$-adic filtration. Then $\gr A=(\gr R)\otimes_K D=D[t]$.
Hence $\gr A$ is a domain and thus so is $L\otimes_K D=L\otimes_R A$. 
Thus $L\otimes_K D$ is a division algebra and therefore
\[
\index(L\otimes_K D)=\deg(L\otimes_K D)=\deg(D)=\index(D)
\]
finishing the proof. 
\end{proof}
\begin{remark} It is easy to see that this theorem and its proof generalize to the
  case that $R$ is regular local ($\gr A$ becomes a polynomial
  ring in several variables). However the theorem is false without the
  assumption of $R$ being regular. Consider the following counter
  example. Let $D$ be arbitrary (of index $>1$) and let $X$ be the Brauer-Severi variety of $D$. It is
  well-known that $X$ has a very ample line bundle $\Lscr$. Let $R$ be
  the local ring of the corresponding cone at the singular vertex.
  Then the quotient field $L$ of $R$ is $k(X)[t]$. Hence $L$ splits
  $D$ and thus $\index(L\otimes_K D)=1<\index(D)$.
\end{remark}
\section{Invariants}
Let $X$ be a scheme. It is classical that
\[
H^1(X,\GG_m)=\Pic(X)
\]
The \emph{cohomological Brauer group} of $X$ is defined as
$H^2(X,\GG_m)_{\text{tors}}$. Thanks to the following theorem we know
that the cohomological Brauer group often coincides with the usual
Brauer group.
\begin{theorem} (Gabber) Assume that $X$ is a quasi-compact, separated scheme
with an ample invertible sheaf (e.g. $X$ is quasi-projective).
Then 
\[
H^2(X,\GG_m)_{\text{tors}}=\Br(X)
\]
\end{theorem}
Gabber's proof of this result is not widely disseminated. However a different
proof was given by de Jong. See \cite{dJ1}. 

Let $\mu_n$ denote the $n$-th roots of unity. Since by our standing hypotheses $n$ is invertible
we have an exact sequence for the etale topology
\[
1\r \mu_n\r \GG_m\xrightarrow{(-)^n}\GG_m\r 1
\]
(this is the Kummer sequence). Taking cohomology we get a short exact sequence
\begin{equation}
\label{ref-4.1-1}
0\r \Pic(X)/n\Pic(X)\r H^2(X,\mu_n)\xrightarrow{[-]} \Br(X)_n\r 0
\end{equation}
We find that the cohomology group $H^2(X,\mu_n)$ is a finer invariant
than $\Br(x)_n$.

Now we look at 
\[
1\r \mu_n\r \Sl_n\r \PGl_n\r 1
\]
Taking into account that \cite[Proof of Thm IV.2.5]{Milne}
\[
H^1(X,\PGl_n)=\Az_n(X)
\]
where $\Az_n(X)$ denotes the isomorphism classes of Azumaya algebras of rank
$n^2$ we obtain a map
\[
\cl(-):\Az_n(X)\r H^2(X,\mu_n)
\]
Thus we have constructed maps
\[
\Az_n(X)\xrightarrow{\cl(-)} H^2(X,\mu_n)\xrightarrow{[-]} \Br(X)_n
\]
We often denote the composition by $[-]$ as well. 
\begin{lemma} \label{ref-4.2-2} The square in the following diagram is commutative up to sign
\[
\begin{CD}
\Az_n(X) @>\cl(-)>> H^2(X,\mu_n)  @>[-]>> \Br(X)_n\\
@A\HEnd_{\Oscr_X}(-)AA @AAA\\
\Vect_n(X) @>>\wedge^n(-)\mod n> \Pic(X)/n\Pic(X)
\end{CD}
\]
where $\Vect_n(X)$ denotes the isoclasses of vector bundles of rank $n$
on $X$. 
\end{lemma}
\begin{proof} This follows from the following commutative diagram of
groups
\[
\begin{CD}
1 @>>> \mu_n @>>> \SL_n @>>> \PGL_n @>>> 1\\
@. @A\Id AA @AA\pr_1 A @AAA\\
1 @>>> \mu_n @>x\mapsto(x,x^{-1})>> \SL_n\times \GG_m @>(u,v)\mapsto uv>> \GL_n @>>> 1\\
@. @V(-)^{-1}VV @VV\pr_2 V @VV\det V\\
1 @>>> \mu_n @>>> \GG_m @>>(-)^n> \GG_n @>>> 1\\
\end{CD}
\]
Taking into account that
\[
H^1(X,\Gl_n)=\Vect_n(X)
\]
and that the map
\[
H^1(X,\Gl_n)\r H^1(X,\PGl_n)
\]
corresponds to 
\[
E\mapsto \End_{\Oscr_X}(E)
\]
(see \cite[Proof of Thm IV.2.5]{Milne}) we obtain a commutative diagram
\[
\begin{CD}
\Az_n(X) @>\cl>> H^2(X,\mu_n)\\
@A\HEnd_{\Oscr_X}(-)AA @|\\
\Vect_n(X) @>>> H^2(X,\mu_n)\\
@V\wedge^n(-) VV @VV\times(-1)V\\
\Pic(X) @>>> H^2(X,\mu_n)
\end{CD}
\]
which yields what we want. 
\end{proof}
\section{Ramification}
\label{ref-5-3}
It is well know that etale cohomology can be computed in terms of
local invariants. We
discuss this in a very special case.

Let $l$ be a prime number different from the characteristic.\footnote{Most
of what we say is valid for $l$ not divisible by the characteristic.} Let
$\mu_l$ denote the $l$'th roots of unity. The following is the version by
Bloch and Ogus of the Grothendieck coniveau spectral sequence
\cite[p164]{dixexp}. 
\begin{theorem} \cite[Prop (3.9)]{BlochOgus}
Assume that $X$ is smooth over a perfect ground field $k$. Then there
is a spectral sequence 
\begin{equation}
\label{ref-5.1-4}
E^{pq}_1=\oplus_{x\in X^{(p)}} H^{q-p}(x,\mu_l^{\otimes n-p})
\Rightarrow 
H^{p+q}(X,\mu_l^{\otimes n})
\end{equation}
Here $X^{(p)}$ are the points of codimension $p$ is $X$. I.e.\ the points
such that $\Oscr_{X,x}$ has (Krull) dimension $p$.
\end{theorem}
 Furthermore $H^p(x,
\mu_l^{\otimes n})$
is defined as 
\[
H^p(x,\mu^{\otimes n}_l)=\dirlim_{U\subset \overline{\{x\}}} H^p(U,\mu^{\otimes n}_l )
\]
where $U$ runs through the open subsets of $\overline{\{x\}}$. Note
that
\[
\dirlim_{U\subset \overline{\{x\}}} H^p(U,\mu_l^{\otimes n} )=
H^p(\Spec k(x),\mu_l^{\otimes n})
\]
To prove this we may assume that $x$ is the generic point of $X$ and that
$X$ is affine. We have $\Spec k(x)=\invlim U$ where $U$ runs through the open
affines of $X$. We can now invoke \cite[Thm VII.5.7]{SGA4} which says
that under these circumstances etale cohomology commutes with inverse limits.

The main point of \cite{BlochOgus} is that the sheaffification for the Zariski
topology of \eqref{ref-5.1-4} degenerates at $E_2$ and computes the sheaffification
of the presheaf $U\mapsto H^p(U,\mu^{\otimes n})$. This will not concern us 
however.

If $K$ is a field transcendence degree $r$ over the algebraically
closed field $k$ then the $l$-cohomological dimension of $K$ is $r$.
See \cite[p223 top]{Milne}. The $l$-cohomological dimension is the degree
above which all Galois cohomology of $l$-torsion sheaves vanishes. This means
that $E^{pq}_1=0$ for $q>\dim X$.

We will now write out explicitly the coniveau spectral sequence in the case
that $X$ is a smooth proper connected surface over an algebraically closed
field $k$.  
\[
\begin{matrix}
  \bigoplus_{x\in X^{(0)}} H^2(x,\mu_l^{\otimes n}) &\bigoplus_{x\in
    X^{(1)}} H^1(x,\mu^{\otimes n-1}_l) & \bigoplus_{x\in X^{(2)}}
  H^0(x,\mu^{\otimes n-2}_l)
  \\*[0.2in]
  \bigoplus_{x\in X^{(0)}} H^1(x,\mu_l^{\otimes n}) & 
\bigoplus_{x\in X^{(1)}} H^0(x,\mu_l^{\otimes n-1})&0\\*[0.2in]
  \bigoplus_{x\in X^{(0)}} H^0(x,\mu_l^{\otimes n}) & 0&0
\end{matrix}
\]
So if  $C$ runs through the irreducible 
curves in $X$ and $x$ runs through the closed points of $x$ we get a complex
\begin{equation}
\label{ref-5.2-5}
H^2(X,\mu^{\otimes n}_l) \r  H^2(k(X),\mu^{\otimes n}_l)
\r \bigoplus_{C} H^1(k(C),\mu^{\otimes n-1}_l)\r \\
\bigoplus_{x} \mu^{\otimes n-2}_l
\r \mu^{\otimes n-2}_l\r 0
\end{equation}
where we have used Poincare duality
\[
H^4(X,\mu^{\otimes n}_l)\cong \mu^{\otimes n-2}_l
\]
This complex is acyclic in positions 2,4,5 and has homology in
$H^3(X,\mu^{\otimes n}_l)$ in position 3.

Since $X$ is a smooth surface we have that $\Br(X)_l\r\Br(k(X))_l$ is 
injective \cite[Cor.\ IV.2.6]{Milne}. From \eqref{ref-4.1-1} applied
to $\Spec k(X)$ we get 
\[
H^2(k(X),\mu_l)= \Br(k(X))_l
\]
Whence \eqref{ref-5.2-5} transforms into a complex 
\begin{equation}
\label{ref-5.3-6}
0\r \Br(X)_l \r  \Br(k(X))_l
\xrightarrow{a} \bigoplus_{C} H^1(k(C),\ZZ/l\ZZ)\xrightarrow{r} \\
\bigoplus_{x} \mu^{-1}_l
\xrightarrow{s} \mu^{-1}_l\r 0
\end{equation}
which is now also exact on the left. 

\medskip

In \cite{AM} an explicit description of the maps in this complex is
given.\footnote{The proof in \cite{AM} for the existence of this sequence is 
a bit different from, but equivalent to the one given here.}
The map $a$ is $\oplus_{C} \Ram_C$ where 
\[
\Ram_C: \Br(k(X))_l\r H^1(k(C),\ZZ/l\ZZ)
\]
is the ``standard ramification map'' associated to the discrete
valuation ring $R\subset K\overset{\text{def}}{=}$ corresponding to
$C\subset X$.  Note that we have
\[
H^1(k(C),\ZZ/l\ZZ)=\Hom(\Gal(k(C)),\ZZ/l\ZZ)
\]
from which one deduces the well-known basic fact that the non-trivial
elements of $H^1(k(C),\ZZ/l\ZZ)$ are represented by couples
$(L/k(C),\sigma)$ where $L/k(C)$ is a cyclic extension of degree $l$
and $\sigma$ is a generator of $\Gal(L/k(C))$ ($\sigma$ is the inverse
image of $\bar{1}\in \ZZ/l\ZZ$). Hence if $\ram_C$ takes a non-trivial
value on $[A]\in \Br(k(X))_l$ (i.e.\ $A$ is ``ramified'' in $C$) then $\ram_C([A])$
defines a couple $(L/k(C),\sigma)$ as above.

\begin{remark} 
  By a local computation one may show that the following ring
  theoretic description of the ramification map is correct \cite{AdJ}. Let $D\in
  \Br(k(X))_l$ be a division algebra ramified in $C$. Let $\Lambda\subset D$ be a
  maximal order over $R$. Let $M$ be the twosided maximal ideal of
  $\Lambda$.  Then $\Lambda/M$ be is a central simple algebra. Let $L$
  be its center. 
  One may show that $M$ is generated by a normalizing element $\Pi$.
  Then $\sigma=\Pi\cdot \Pi^{-1}$ defines an automorphism of $L$ over
  $k(C)$ and $\ram_C([D])=(L/k(C),\sigma)$. See \cite{Reiner} for some of
the unexplained terminology in this remark. 
\end{remark}

\medskip

The map $r$ is a direct sum $\oplus_{C,x} r_{C,x}$ where $x\in C$. Assume
that $C$ is smooth and $(L/k(C),\sigma)\in H^1(k(C),\ZZ/l\ZZ)$. Then 
$r_{C,x}(L/k(C),\sigma)$ measures the ramification of $L/k(C)$ in $x$ in the
sense that $r_{C,x}(L/k(C),\sigma)=0$ if and only if $L/k(C)$ is unramified
in $x$ (in the classical sense of extensions of discrete valuation rings). 
\begin{remark}
Using a local computation we can show that the following is correct. 
Let $K=k(C)$ and let $K_s$ be the separable closure of $K$. Then
we have an exact sequence
\[
0\r \mu_l\r K^\ast_s\xrightarrow{l} K_s^\ast\r 0
\]
from which we deduce using Hilbert's theorem 90 (recall that over a field
etale cohomology is the same as Galois cohomology). 
\[
H^1(K,\mu_l)=K^\ast/(K^\ast)^l
\]
If $v$ is a discrete valuation (associated to a point $x\in C$) then
we obtain an induced map
\[
r:H^1(K,\mu_l)\r\ZZ/l\ZZ:\bar{s}\mapsto \overline{v(s)}
\]
Since $\mu_l\cong \ZZ/l\ZZ$ (non-canonically) we may twist $r$ to obtain
a (canonical) map
\[
r':H^1(K,\ZZ/l\ZZ)\r\mu_l
\]
Then one has $r'=r_{C,x}$. 
\end{remark}
\medskip

The map $s$ is simply the sum map. Note that for arbitrary $C$ we have a complex
\[
H^1(k(C),\ZZ/l\ZZ)\r \bigoplus_{x\in C} \mu^{-1}_l
\xrightarrow{s} \mu^{-1}_l
\]
In other words is $\xi\in H^1(k(C),\ZZ/l\ZZ)$ then
\[
\sum_{x\in C} r_{C,x}(\xi)=0
\]
This is a kind of reciprocity law. 

\medskip

The only reason why we need the coniveau complex is the following
\begin{proposition} \label{ref-5.4-7} Assume that $0\neq [A]\in \Br(X)_l$ where $X$ is a smooth connected
surface over an algebraically closed field $k$. Then it is impossible for $A$ to
be
ramified in a union of trees of $\PP^1$'s. 
\end{proposition}
\begin{proof}
  Assume that the ramification locus of $A$ is a union of trees of
  $\PP^1$'s. Then there is at least one of those $\PP^1$'s which
  intersects the other components of the ramification locus in at most
  one point.  Denote this $\PP^1$ by $C$. The case that there is no
  intersection point is very slightly easier so assume there is an
  intersection point and denote it by $x$. Finally put
  $\xi=\Ram_C(A)$.  If $x\neq y\in C$ then there is no other component
  of the ramification locus which intersects $y$ and hence it follows
  from \eqref{ref-5.3-6} that $r_{C,y}(\xi)=0$.

From the identity $\sum_{x\in C}r_{C,x}(\xi)=0$ we then obtain $r_{C,x}(\xi)=0$.
But this means that $\xi$ is unramified over $\PP^1$. To finish the proof
we note that $\PP^1$ does not have any unramified coverings. See e.g.\ \cite[I.5(f)]{Milne}.
\end{proof}
A particular type of surface singularity is a so-called rational
double point.  The minimal resolutions of such singularities are trees
of $\PP^1$'s whose incidence graph is a Dynkin diagram. See
\cite{Lipman2}. We obtain the following corollary to Proposition
\ref{ref-5.4-7}. 
\begin{corollary} 
\label{ref-5.5-8}
Assume that $X$ is a proper, connected singular surface
whose only singularities are rational double points. Let $U$ be
the complement of the singular locus and let $\tilde{X}\r X$ be the
minimal resolution of singularities. If $\alpha\in \Br(k(X))_l$ is unramified
on $U$ then $\alpha$ is unramified on $\tilde{X}$ as well. 
\end{corollary}
\begin{proof} If $\alpha$ is ramified on $\tilde{X}$ then it can be at
  most be ramified on the exceptional locus which is a union of trees
  of $\PP^1$'s. According to Proposition \ref{ref-5.4-7} it must then be
  unramified.
\end{proof}
The simplest rational double points are $A_{l-1}$ singularities. They
can be characterized as those singularities whose local ring $(R,m)$ after
completion becomes isomorphic to 
\[
k[[x,y,z]]/(z^l-xy)
\]
($l$ does not have to be prime the the characteristic here). An easy local
computation shows that such an $A_{l-1}$ singularity can be desingularized
by $\lceil (l-1)/2\rceil$ consecutive blowups at singular points (which
are themselves $A_{l'-1}$ singularities for $l'\le l$). The resulting
exceptional locus is the Dynkin diagram $A_{l-1}$.
\begin{example} \label{ref-5.6-9}
Let $(R,m)$ be a regular local $k$-algebra of dimension two with
  residue field $k$ and  let
$f\in m^2$ be such that $f-xy\cong 0 \mod m^3$ where $m=(x,y)$. 
Then $R=R[w]/(w^l-f)$ is an $A_{l-1}$-singularity.
To see this note that after completion we have $R=k[[x,y]]$. We claim
that after redefining $x$, $y$ we may assume $f=xy$. Assume
\[
f-xy\in m^t
\]
Put $x'=x+p$, $y'=y+q$ such that $p,q\in m^t$. Then
\[
f-x'y'\cong(f-xy)-xq-yp\mod m^{t+1}
\]
Since $x,y$ generate $m$ we may find $p,q$ such that $f-x'y'\in m^{t+1}$.
Repeating this procedure proves the claim.
\end{example}
Assume that $X$ is a smooth proper connected surface and let $\Lscr\in \Pic(X)$. Let $0\neq s\in 
H^0(X,\Lscr^l)$. We define a sheaf of algebras on $X$ as follows
\[
\Ascr=\bigoplus_{i=0}^{l-1}\Lscr^{-i}T^i
\]
where $T^l=s$ and $Y=\underline{\Spec}\Ascr$. Then $Y\r X$ is a cyclic 
covering of degree $l$ of $X$ which is ramified in the zeroes of $s$. 
\begin{proposition}
\label{ref-5.7-10}
  Assume that the zero divisor $H$ of $s$ has normal crossings with smooth components. Let
  $\tilde{Y}$ be the minimal resolution of singularities of $Y$. Let
  $\alpha\in \Br(k(X))_l$ and assume that the ramification locus of
  $\alpha$ is contained in $H$. Then $\alpha_{k(\tilde{Y})}$ is
  unramified.
\end{proposition}
\begin{proof}  
By Example \ref{ref-5.6-9} $Y$ has $A_{l-1}$-rational singularities.
As the ramification locus of $\alpha$ is contained in
$H$ it is standard that $\alpha_{k(Y)}$ is unramified on the regular
part of $Y$ (e.g.\ \cite[\S23.3]{CT1}). But then it follows from Corollary 
\ref{ref-5.5-8} that 
$\alpha_{k(\tilde{Y})}$ is unramified as well.
\end{proof}
Before moving one we state a result which is peculiar to
dimension two. 
\begin{lemma} \label{ref-5.8-11} Let $X$ be a smooth projective surface and let $A$ be
a central simple algebra over $k(X)$ which is unramified. Then there
exists an Azumaya algebra $\Ascr$ on $X$ such that $\Ascr_{k(X)}=A$.
\end{lemma}
\begin{proof} Let $\Ascr$ be a maximal order in $A$. I.e. a sheaf of
  $\Oscr_X$-algebras torsion free and coherent as $\Oscr_X$-module
  which is contained in  $A$ (such an object is called an ``order'') and
  which is not properly contained in any other order. Maximal orders are
  a non-commutative version of integrally closed rings but they are not
  unique \cite{Reiner}.

As an $\Oscr_X$ module we have 
\[
\Ascr=\Ascr^{\vee\vee}=\uHom_X(\uHom_X(\Ascr,\Oscr_X),\Oscr_X)
\]
since otherwise $\Ascr^{\vee\vee}$ would be a bigger order. A standard
fact in commutative algebra says that a reflexive module over a ring
of global dimension two is projective. So $\Ascr$ is locally free.
Furthermore since $A$ was supposed to be unramified we have that
$\Ascr_x$ is Azumaya for all $x\in X^{(1)}$ where $X^{(1)}$ denotes the
points of codimension one. 

Consider now the standard map
\[
\Ascr\otimes_{\Oscr_X} \Ascr^\circ \r \uEnd_{\Oscr_X}(\Ascr)
\]
This now an isomorphism in codimension one between locally free modules
of the same rank. Hence it has to be an isomorphism. 
\end{proof}
\section{Artin splitting}
If $\Ascr$ is an Azumaya algebra of index $d$ on a scheme $X$, $\Lscr$
is an invertible sheaf on $X$ and $s\in H^0(X,\Ascr\otimes_X \Lscr)$ then we
denote the (reduced) characteristic polynomial of $s$ by
$P_s(T)$. The coefficient of $T^i$ is a section of
$\Lscr^{\otimes d-i}$.

Put
\[
L=\underline{\Spec} \oplus_n \Lscr^{-n}T^n
\]
We may view $P_s(t)\Lscr^{-d}$ as sections of $\oplus_n
\Lscr^{-n}T^n$. Hence inside $L$ they cut out a subscheme which we
denote by $Y_s$. Looking e.g.\ locally it is easy to see that $Y_s/X$
finite and flat. 

 If we do not specify $\Lscr$ (e.g.\ in the local
case) then we assume $\Lscr=\Oscr_X$ (or rather that we have chosen
some unspecified trivialization $\Lscr\cong \Oscr_X)$. 

The following result is due to  Mike Artin \cite{AdJ}
\begin{theorem} 
\label{ref-6.1-12}
Assume that $X$ is a smooth projective surface over an algebraically
closed field $k$. Let $\Ascr$ be an Azumaya algebra of degree $d$ over
$X$ which is generically a division algebra.  Let $\Mscr$ be an ample
line bundle on $X$. Then for $r\gg 0$ and $\Lscr=\Mscr^r$ there exists
a Zariski open subset $U\subset H^0(X,\Ascr\otimes_{\Oscr_X}\Lscr)$
such that for $s\in U$ we have
\begin{enumerate}
\item $Y_s$ is a smooth connected surface.
\item $Y_s/X$ is generically etale (i.e. $k(Y_s)/k(X)$ is separable). 
\item $Y_s$ splits $\Ascr$.
\end{enumerate}
\end{theorem}
\begin{remark} This can be extracted from \cite[Thm 8.1.11]{AdJ}. However this
last result is much more precise and it does not assume that $\Ascr$ is generically
a division algebra. 
\end{remark}
The following key lemma analyzes the local situation.
\begin{lemma}
\label{ref-6.3-13}
Let $(R,m)$ be a regular
local $k$-algebra of dimension $2$ with residue field $k$. Let $A$ be an Azumaya algebra of degree $d$
over $R$.  Then there exists a closed subvariety $W$ in $A/m^2A$ (the
latter viewed as a vector space) of codimension 3 such that for $s\in A$
we have that $Y_s $ is regular if $\bar{s}\not\in W$ (here and below $\bar{s}$ stands
for $s\mod m^2$).
\end{lemma}
\begin{proof} We choose an isomorphism $A/m^2A\cong M_d(R/m^2)$ and
an isomorphism $R/m^2\cong k[\eta,\zeta]/(\eta^2,\zeta^2,\eta\zeta)$.
Using these isomorphisms we let $\GL_d(k)$ act on $A/m^2A$ by conjugation.

We let $W$ be the of the locus of elements
$\bar{s}=s_0+s_1\eta+s_2\zeta$ in $A/m^2A$ such that $s_0$ has a repeated eigenvalue
$\mu$ 
with the property $\det(\bar{s}-\mu)=0$. We claim this is a closed subset of 
$A/m^2$.
To prove this 
consider the subset $\widetilde{W}$ of $\PP((A/m^2)\oplus k)$  defined by
\[
\widetilde{W}=
\{(\bar{s},\mu)\mid \text{$\mu$ is a repeated eigenvalue of $s_0$
and $\det(\bar{s}-\mu)$}=0\}
\]
This is closed since the condition that $\mu$ is a repeated eigenvalue can be expressed
as 
\[
\rk (s_0-\mu)^2\le d-2
\]
Hence $\widetilde{W}$ is a projective scheme.
The projection map $\widetilde{W}\r \PP(A/m^2A): (\bar{s},\mu)\mapsto
\bar{s}$ is well defined  since if $\bar{s}=0$ then
$\det(\mu)=0$ and hence $\mu=0$.  Hence
the image of $\widetilde{W}$ in $\PP(A/m^2A)$  is closed.  Since $W$ is the cone 
 of the image of $\widetilde{W}$ in $\PP(A/m^2A)$ it follows that $W$ is
also closed.

We can describe the elements of $W$ by putting $s_0$ in
Jordan normal form.  One verifies that $\bar{s}\in W$ if and only if one
of the following conditions holds
\begin{enumerate}
\item $s_0$ has two Jordan blocks with equal eigenvalues. 
\item $s_0$ has a Jordan block of size $>1$ of the form
\[
\begin{pmatrix}
\mu &1 & \cdots&0 & 0\\
0&\mu &\cdots&0 &0\\
\vdots & \vdots& &\vdots&\vdots\\
0 &0&\cdots&\mu & 1\\
\ast &0&\cdots &0&\mu
\end{pmatrix}
\]
such that $\bar{s}_{1,ij}=\bar{s}_{2,ij}=0$ where $(i,j)$ refers to
the entry of $\bar{s}$ marked by ``$\ast$'' in the above matrix.
\end{enumerate}
From this description it follows easily that $W$ has codimension three.

\medskip

Assume now $s\in A$ such that $\bar{s}\not\in W$. Denote the closed point
of $\Spec R$ by $o$.

We will the describe
the equation of $Y_s$ in the neighborhood of a closed point $(o,\mu)\in
Y_s$ where $\mu$ is an eigenvalue of $s_0$. Replacing
$s$ by $s-\mu$ we may assume $\mu=0$. The equation of $Y_s$ around
$(o,0)$ will be
\[
P_s(T)=T^d+a_1T^{d-1}+\cdots +a_{d-1}T+a_d
\]
where $a_i\in R$. We have
\begin{align*}
a_d&\cong \pm\det(\bar{s})\mod m^2\\
a_{d-1}&\cong \pm\operatorname{tr}(\operatorname{ad}(\bar{s}))\mod m^2
\end{align*}
%
To be able to compute we put $s_0$ in Jordan normal form.
We separate two cases.
\begin{enumerate}
\item The eigenvalue $0$ has multiplicity one. 
In that case $a_{d-1}\in R^\ast$.
\item The eigenvalue $0$ has higher multiplicity. Since $\bar{s}\not
\in W$ we obtain from the definition of $W$: $\bar{a}_d=\det(\bar{s})\neq 0$. In
particular $a_d\not\in m^2$.
\end{enumerate}
Hence in both cases $P_{s}(T)\not\in (m,T)^2\subset R[T]_{(T)}$. Thus $Y_s$
 is regular in $(o,0)$.
\end{proof}
We need a trivial lemma about ample linebundles.
\begin{lemma}
\label{ref-6.4-14}
Let $\Mscr$ be ample on $X$ and let $\Fscr$ be a coherent sheaf on
$X$.  Then for $r\gg 0$ and for $\Lscr=\Mscr^r$ we have
that for all $x$ the map 
\[
H^0(X,\Fscr\otimes_X \Lscr)\r H^0(X,\Fscr/m_x^2\Ascr\otimes_X \Lscr)
\]
is surjective. 
\end{lemma}
\begin{proof}
 For $s\gg 0$ we have that $\Mscr^s$ is very
ample \cite{H}, i.e.\  for all $x\in X$ we
have that
\[
H^0(X,\Mscr^s)\r H^0(X,\Mscr^s\otimes \Oscr_X/m^2_x)
\]
is surjective. See \cite[Prop 7.3]{H}. 

For $t\gg 0$ we have that $\Fscr\otimes \Mscr^t$ is generated by global
sections \cite{H}. I.e. there is a surjective map
\[
\Oscr_X^{\oplus N}\r \Fscr\otimes \Mscr^t
\]
and hence surjective maps
\[
(\Mscr^s_X)^{\oplus N}\r \Fscr\otimes \Mscr^{s+t}
\]
\[
(\Mscr^s_X\otimes \Oscr_X/m_x^2)^{\oplus N}\r \Fscr\otimes \Mscr^{s+t}\otimes \Oscr_X/m_x^2
\]
Since the last map is between sheaves concentrated in a point, it remains
surjective after applying $H^0(X,-)$. 

The lemma now follows from the following commutative diagram
\[
\begin{CD}
H^0(X,\Mscr^s)^{\oplus N} @>>> H^0(X,\Fscr\otimes \Mscr^{s+t})\\
@VVV @VVV\\
H^0(X,\Mscr^s\otimes \Oscr_X/m^2_x)^{\oplus N} @>>> H^0(X,\Fscr\otimes \Mscr^{s+t}\otimes \Oscr_X/m^2_x)
\end{CD}
\]
using the fact that the lower horizontal and the leftmost vertical maps
are surjective.
\end{proof}
\begin{proof}[Proof of Theorem \ref{ref-6.1-12}] 
We prove smoothness of $Y_s$ first. By Lemma \ref{ref-6.4-14} we have that for $r\gg 0$ the map
\[
p:H^0(X,\Ascr\otimes_X \Lscr)\r H^0(X,\Ascr/m_x^2\Ascr\otimes_X \Lscr)
\]
is surjective for all $x$. Let $W_x\subset H^0(X,\Ascr/m_x^2\Ascr\otimes_X \Lscr)$
be the subvariety we denoted by $W$ in Lemma \ref{ref-6.3-13} (choosing
an arbitrary trivialization of $\Lscr$ in a neighborhood of $x$).

Let $s_1,\ldots, s_n$ be a basis for $H^0(X,\Ascr\otimes_X \Lscr)$.
Put $\AA^n=\Spec k[t_1,\ldots,t_n]$ and put $\tilde{s}=t_1s_1+\cdots+t_ns_n$. 
Let $\psi:Y_{\tilde{s}}\r X\times \AA^n$ be the ramified cover defined by
$\tilde{s}$. \def\NSm{\operatorname{NSm}} Let $\NSm(Y_{\tilde{s}})$ be the locus where the projection
$Y_{\tilde{s}}\r \AA^n$ is not smooth (i.e.\ where $\Omega_{Y_{\tilde{s}}/\AA^n}$
is not locally free). We need to prove that the image
of $\NSm(Y_{\tilde{s}})$ is not $\AA^n$.

For $x\in X$ we have shown above that $\psi^{-1}(x\times \AA^n)\cap
\NSm(Y_{\tilde{s}})\subset \psi^{-1}(x\times p^{-1}(W_x))$. It follows
that $\NSm(Y_{\tilde{s}})$ has codimension $\ge 3$ in $Y_{\tilde{s}}$.
Since $\dim X=2$ this means that the image of $\NSm(Y_{\tilde{s}})$ cannot be
the full $\AA^n$. This finishes the proof of smoothness.

Now we prove splitting and generic etaleness. Fix $x\in X$. There is a
Zariski open $\bar{V}\subset A/mA\cong M_d(k)$ such that for $t\in
\bar{V}$ we have that the characteristic polynomial of $t$ has $d$
distinct roots. Let $V$ be the inverse image of $\bar{V}$ under the surjective
map
\[
p':H^0(X,\Ascr\otimes_X \Lscr)\r H^0(X,\Ascr/m_x\Ascr\otimes_X \Lscr)
\]
Hence $Y_s/X$ is unramified in $x$ if $s\in V$.  Since $Y_s/X$ is flat we deduce
that $Y_s/X$ is etale in a neighborhood of $x$. Hence $Y_s/X$ is
generically etale if $s\in V$.

If $k(Y_s)$ does not split $D=\Ascr_\eta$ then the map
\[
k(Y_s)\r D:T\mapsto s
\]
must land in a non-maximal subfield of $D$. But then $P_s(T)$ is not
the minimal polynomial of $s$ and hence it is has multiple roots. It
follows that $Y_s/X$ is not generically etale.
\end{proof}
\section{Elementary transformations}
Suppose $X$ is a scheme
and $I$ is an invertible ideal in $\Oscr_X$.  Let $\Ascr$ be an Azumaya algebra on
$X$ and assume that $\bar{\Ascr}=\Ascr/I\Ascr=\End_{\Oscr_D}(\bar{V})$ with $\bar{V}$ 
a vector bundle on  $D=V(I)$.  Suppose we have a
subbundle $\bar{F}$ of $\bar{V}$. 

We define
\[
\bar{\Bscr}=\{\phi\in \End_{\Oscr_D}(\bar{V})\mid \phi(\bar{F})\subset \bar{F}\}
\]
and we let $\Bscr$ be the inverse image of $\bar{\Bscr}$ in $\Ascr$. Define
\[
\bar{J}=\{\phi\in \End_{\Oscr_D}(\bar{V})\mid \phi(\bar{V})\subset \bar{F}\}
\]
and let $J$ be the inverse image of $\bar{J}$ in $\Ascr$. Thus we have
inclusions
\[
I\Ascr\subset J\subset \Bscr\subset \Ascr
\]
It is clear that $J$ is a twosided ideal in $\Bscr$. It is also a right
$\Ascr$ ideal. We put 
\[
\Ascr'=\End_\Ascr(J)
\]
and call $\Ascr'$ the \emph{elementary transform} of $\Ascr$ with respect to the data 
$(D,\Fscr,\Vscr)$. 
\begin{remark} Max Lieblich points out that this definition can be generalized and in
this way
actually becomes more transparent. Instead of giving $\Vscr$ and $\Fscr$ as input
one may start directly with a locally free right ideal $\bar{J}\in \bar{\Ascr}$ which is locally
a direct summand. In this way we do not have to assume that $\bar{\Ascr}$ is split. 
In these notes we stick to the original definition.
\end{remark}
To analyze the properties of elementary transformations we may work
locally for the etale topology. Switching to an affine setting with
$X=\Spec R$ we may assume that $\Ascr=M_n(R)$.  $\bar{V}=\bar{R}^n$,
$\bar{F}=\bar{R}^k$. Then we find
\[
J=
\begin{pmatrix}
I & \cdots & I  &   I & \cdots & I \\
\vdots & & \vdots  &  \vdots & & \vdots\\
I & \cdots &I  & I & \cdots & I \\
R & \cdots& R &   R & \cdots & R         \\
\vdots & & \vdots  &  \vdots & & \vdots        \\
R & \cdots & R  & R & \cdots & R  
\end{pmatrix}
\]
\[
\Bscr=\begin{pmatrix}
R & \cdots & R  &   I & \cdots & I \\
\vdots & & \vdots  &  \vdots & & \vdots\\
R & \cdots &R  & I & \cdots & I \\
R & \cdots& R &   R & \cdots & R         \\
\vdots & & \vdots  &  \vdots & & \vdots        \\
R & \cdots & R  & R & \cdots & R           
\end{pmatrix}
\]
\[
\Ascr'=\begin{pmatrix}
R & \cdots & R  &   I & \cdots & I \\
\vdots & & \vdots  &  \vdots & & \vdots\\
R & \cdots &R  & I & \cdots & I \\
I^{-1} & \cdots& I^{-1} &   R & \cdots & R         \\
\vdots & & \vdots  &  \vdots & & \vdots        \\
I^{-1} & \cdots & I^{-1}  & R & \cdots & R           
\end{pmatrix}
\]
We deduce immediately
\begin{lemma} The elementary transform of an Azumaya algebra is an Azumaya algebra. 
\end{lemma}
Now we work globally again.  Let $\bar{Q}=\bar{V}/\bar{F}$. Directly
from the definition we deduce that there is a short exact sequence
\[
0\r \Bscr\r \Ascr\r \Hom_{\bar{R}}(\bar{F},\bar{Q})\r 0
\]
We will now construct a similar exact sequence for $\Ascr'$.  The inclusion
$I\Ascr\subset J$ yields a map
\[
\Ascr'=\Hom_{\Ascr}(J,J)\r I^{-1} J=I^{-1}\otimes J\r I^{-1}\otimes_{\Oscr_D} \bar{J}
\]
One checks locally that the image of this map is precisely
\[
I^{-1}\otimes_{\Oscr_D} \Hom_{\Oscr_D}(\bar{Q},\bar{F})
\]
and likewise one checks locally that the kernel of the induced
map
\[
\Ascr'\r I^{-1}\otimes_{\Oscr_D} \Hom_{\Oscr_D}(\bar{Q},\bar{F})
\]
is precisely $\Bscr$. Taking into account that
\[
I^{-1}\otimes_{\Oscr_D}
\Hom_{\Oscr_D}(\bar{Q},\bar{F})=\Hom_{\Oscr_D}(I/I^2\otimes_{\Oscr_D} \bar{Q},
\bar{F})
\]
we get an exact sequence
\[
0\r \Bscr \r \Ascr'\r \Hom_{Oscr_D}(I/I^2\otimes_{\Oscr_D}\bar{Q},
\bar{F})\r 0
\]
There is a proof of the following proposition in \cite{AdJ}. 
We give the proof in \cite{CT1}.
\begin{proposition}
  Let $X$ be a scheme, $\Ascr\in \Az_n(X)$, with $n$ invertible
  on $X$ and $I\subset \Oscr_X$ an invertible ideal. Assume there is a
  $\Oscr_D=\Oscr_X/I$ vector bundle $\bar{V}$ such that
$\bar{\Ascr}=\Ascr/I\Ascr=\HEnd(\bar{V})$. Let $\bar{F}\subset \bar{V}$ be a subbundle
of constant rank $r$. Let $\Ascr'$ be the associated elementary transform. Then
\begin{equation}
\label{ref-7.1-15}
\cl(\Ascr')=\cl(\Ascr)+r [I] \in H^2(X,\mu_n)
\end{equation}
where $[I]$ is the image of $I$ under the composition
\[
\Pic(X)\r \Pic(X)/n\r H^2(X,\mu_n)
\]
\end{proposition}
\begin{proof} We first consider the case that $\Ascr=\End_X(V)$ is split,
$\bar{V}=V/IV$ and $\bar{F}\subset \bar{V}$.  Let $\bar{Q}=\bar{V}/\bar{F}$
and let $W$ be the kernel of $V\r \bar{Q}$.  Then 
$J=\Hom_X(V,W)$ and hence $\Ascr'=\End_X(W)$.

Hence according to Lemma \ref{ref-4.2-2} $\cl(\Ascr)=-[\wedge^n V]$,
$\cl(\Ascr')=-[\wedge^n W]$. So we have to compare $\wedge^n V$ and
$\wedge^n W$.  The inclusion $W\subset V$ certainly yields an
inclusion $i:\wedge^n W\subset \wedge^n V$.

 Working locally we may assume $V=\Oscr_X^n$
and $W=I^{n-r}\oplus \Oscr^{r}_X$. We then find that the image of $i$ is equal
to $I^{n-r} (\wedge^n V)$. Hence
\[
[\wedge^n W]=(n-r)[I]+[\wedge^n V]\cong -r[I]+[\wedge^n V]
\]
(we are working modulo $n$). Hence  by lemma \ref{ref-4.2-2}
\[
\cl(\Ascr')=r[I]+\cl(\Ascr)
\]
Before we continue we mention that we could have taken $\bar{V}=J\otimes V/IV$,
$J\in \Pic(D)$
and $\bar{F}\subset J\otimes V/IV$, but then the elementary transform
with respect to $\bar{F}$ is the same as the elementary transform with
respect to $J^{-1}\otimes \bar{F}\subset V/IV$ and $\rk_{\bar{R}}(\bar{F})=
\rk_{\bar{R}}(J^{-1}\otimes \bar{F})$.
 So our assumption that $\bar{V}=V/IV$ was not a restriction.

\medskip

 Now we assume that $\Ascr$ is general.  The next argument is due to
 Gille. Let $\pi:Y\r X$ be the Brauer-Severi associated to $\Ascr$. Thus
$Y$ splits $\Ascr$ and etale locally $Y$ is a $\PP^{n-1}$ bundle over $X$. 

It is
 easy to see that elementary transform is compatible with base change.
 Thus $\pi^\ast (\Ascr')$ is the elementary transform of $\pi^\ast \Ascr$ with
 respect to $\pi^\ast{\bar{F}}$. Since $\pi^\ast(\Ascr)$ is split we find
\[
\pi^\ast \cl(\Ascr')=\cl(\pi^\ast(\Ascr'))=r[\pi^\ast I]+\cl(\pi^\ast(\Ascr))=r\pi^\ast [I]+\pi^\ast(\cl(\Ascr))
\]
Hence it sufficient to show that
\[
\pi^\ast:H^2(X,\mu_n)\r H^2(Y,\mu_n)
\]
is injective. We compute $H^2(Y,\mu_n)$ using the Leray spectral sequence
\[
H^p(X,R^q \pi_\ast \mu_n)\Rightarrow H^{p+q}(Y,\mu_n)
\]
We have $R^0\pi_\ast \mu_n=\mu_n$, $R^1\pi_\ast \mu_n=0$ (see below). Hence
the $E_2$-page of the spectral sequence is as follows
\[
\begin{matrix}
H^0(X,R^2\pi_\ast \mu_n) & H^1(X,R^2\pi_\ast \mu_n)&H^2(X,R^2\pi_\ast \mu_n)\\
0& 0&0\\
H^0(X, \mu_n) & H^1(X, \mu_n)&H^2(X, \mu_n)
\end{matrix}
\]
so that  we have an exact sequence
\[
0\r H^2(X, \mu_n)\r H^2(Y,\mu_n)\r H^0(X,R^2\pi_\ast \mu_n)
\]
and in particular we get the requested injectivity.
\end{proof}
\begin{lemma} Let $Y\r X$ be a relative Brauer-Severi scheme. Then we have
\begin{align*}
R^0\pi_\ast \mu_n&=\mu_n\\
R^1\pi_\ast \mu_n&=0
\end{align*}
\end{lemma}
\begin{proof} Using the proper base change theorem \cite[IV-1]{arcata}
it suffices to
prove this in case $X$ is a geometric point. But then $Y=\PP^{n-1}$
and the cohomology of projective space is well-known (e.g \cite[Example VI.5.6]{Milne}). 
\end{proof}
\begin{remark} In \cite{CT1} the formula \eqref{ref-7.1-15} has a $-$ sign. This
  is because it is assumed $I=\Oscr_X(-D)$ for a Cartier divisor $D$
  and the formula is in terms of $D$.
\end{remark}
\begin{remark} Artin and de Jong \cite{AdJ} show that two Azumaya algebras on a
  surface which are birational can be transformed into each other
by an elementary transform based on a smooth curve. 
\end{remark}
\section{Killing obstructions}
The following we take from de Jong \cite[lemma 3.1]{dJ}.
\begin{lemma} Let $X\r X'$ be a closed immersion defined by an
ideal $I$ of square zero. Let $\Ascr$ be an Azumaya algebra
on $X$. If $H^2(X,I\otimes_X (\Ascr/\Oscr_X))=0$ then $A$ lifts to $X'$.
\end{lemma}
\begin{proof} Etale locally $\Ascr$ is a matrix algebra. A matrix
algebra can obviously be lifted. Furthermore any such lifting is itself
a matrix algebra (since we may lift idempotents). 

It follows that etale locally $\Ascr$ lifts uniquely up to isomorphism.
The sheaf of $\Oscr_{X'}$ isomorphisms of a given lift reducing to
the identity on $\Ascr$ is $I\otimes_X \Der_{\Oscr_X}(\Ascr,\Ascr)$. 
It is part of the standard formalism of deformation theory that the obstructions to gluing local lifts are lying
in the $H^2$ of this sheaf.

We finish by observing that there is an exact sequence
\[
0\r \Oscr_X\r \Ascr\r \Der_{\Oscr_X}(\Ascr,\Ascr)\r 0
\]
This is basically Skolem-Noether for derivations. 
\end{proof}
\begin{remark} Etale cohomology and Zariski cohomology coincide for
  quasi-coherent sheaves \cite[Remark 3.8]{Milne}. So for computing
the cohomology group $H^2(X,I\otimes_X (\Ascr/\Oscr_X))$ we may use
Zariski cohomology if we want to. 
\end{remark}
\begin{lemma}
  \label{ref-8.3-16} Let $D/k$ be a smooth connected projective
  curve and $\Vscr$ a vector bundle of rank $n\ge 2$ on $D$. Let
  $T\subset D(k)$ be finite and assume given for every $t\in T$ a one
  dimensional subspace $F_t\subset V_t$. Then there exists a sub line
  bundle $\Fscr\subset \Vscr$ such that for $t\in T$ we have
  $\Fscr_t=F_t$. Moreover, fixing the other data, $\Fscr$ can be chosen
  to be of arbitrarily negative degree.
\end{lemma}
\begin{proof} Define $\Wscr$ by
\[
0\r \Wscr\r \Vscr\r \oplus_{t\in T} i_{t,\ast}(V_t/F_t)\r 0
\]
Choose a line bundle $\Fscr$ such that $\Wscr\otimes \Fscr^{-1}$ is
generated by global sections. Let $V=\Hom_X(\Fscr,\Wscr)$. We have that the following map is surjective for all $x\in X(k)$
\[
V=H^0(X,\Wscr\otimes_X \Fscr^{-1})\r \Wscr_x\otimes_k \Fscr_x=\Hom_k(\Fscr_x,
\Wscr_x)
\]
In particular $\dim V\ge n$. 

An element $s\in V$ defines a map $\Fscr\r \Wscr\r \Vscr$ as in the statement
of the lemma if the following two conditions hold
\begin{enumerate}
\item If $x\not\in T$ then $s_x$ is not zero.
\item If $x\in T$ then the composition $\Fscr_x\xrightarrow{s_x} \Wscr_x\r F_x$
is not zero.
\end{enumerate}
Let $d=\dim V$. Condition (2) is true on a Zariski open $U$ of $V$
($U$ is the complement of a finite number of hyperplanes). Let
$Z\subset (D-T)\times U$ be the set op pairs $(x, s)$ such that
$s_x=0$. Clearly $Z$ is closed. The dimension of the fibers $Z\r D-T$ is $d-n$.
Thus $\dim Z\le 1+d-n$. Since $n\ge 2$ this is less then $\dim U=d$. It
now suffices to select an $s$ not in the image of the projection
$Z\r U$. 
\end{proof}
\begin{theorem}
\label{ref-8.4-17}
  Let $X/k$ be a smooth connected proper surface. Let $\Ascr\in \Az_n(X)$. Then there is an 
elementary transformation  $\Ascr'\in \Az_n(X)$ of $\Ascr$ such
that
\[
\cl(\Ascr)=\cl(\Ascr')
\]
and 
\[
H^2(X,\Ascr'/\Oscr_X)=0
\]
\end{theorem}
\begin{proof} Using the reduced trace map we see that $\Ascr'/\Oscr_X$
is self dual. Hence it is sufficient construct $\Ascr'$ in such a way that
\[
H^0(X,\Ascr'/\Oscr_X\otimes_X \omega_X)=0
\]
For an arbitrary line bundle $\Lscr$ on $X$ we will construct $\Ascr'$ in
such a way that
\[
H^0(X,\Ascr'/\Oscr_X\otimes_X \Lscr)=0
\]
or equivalently (since $\Oscr_X$ is a factor of $\Ascr$):
\[
H^0(X,\Lscr)\r H^0(X,\Ascr'\otimes_X\Lscr)
\]
is an isomorphism. 


Assume $s\in H^0(X,\Ascr\otimes_X\Lscr)\setminus H^0(X,\Lscr)$. There
exists $t\in X$ such that $s_t\not\in \Lscr_t$. We have
$\Ascr_t=\End_k(V_t\otimes \Lscr_t)$ for $V_t$ a $n$-dimensional
vector space. Hence there is a one dimensional subspace $F_t\subset
V_t$ such that $s_t(F_t)\not\subset F_t\otimes \Lscr_t$.  If $s'\in
H^0(X,\Ascr\otimes_X\Lscr)\setminus H^0(X,\Lscr)$ is such that
$s'_t(F_t)\subset F_t\otimes \Lscr_t$ then we can find $t'\in X$,
$F_{t'}\subset V_{t'}$ such that $s'_{t'}(F_{t'})\not\subset F_{t'}\otimes
\Lscr_{t'}$.

Continuing we find a finite subset $T\subset X$ and one-dimensional
subspaces $(F_t\subset V_t)_{t\in T}$
\[
 H^0(X,\Lscr)=\{s\in H^0(X,\Ascr\otimes_X \Lscr)\mid \forall t\in T: s_t(F_t)
\not\subset F_t\otimes \Lscr_t\}
\]
Choose an ample line bundle $\Lscr$ on $X$. For $q$ sufficiently big a
generic section of $\Lscr^{\otimes nq}$ vanishing on $T$ will define
a smooth curve $D\subset X$ passing through $T$ (this is a version of
Bertini, see below). Thus we have
$\Lscr^{\otimes nq}\cong \Oscr_X(D)$. In particular the class of $D$ in
$\Pic(X)$ is divisible by $n$. 

By Tsen's theorem there is a vector bundle $\bar{V}$ of rank $n$ on
$D$ such that $\Ascr_D=\HEnd_D(\bar{V})$.  For $t\in T$ we have
now two isomorphisms $\Ascr_t\cong\End(V_t)$, $\Ascr_t\cong \End(\bar{V}_t)$. These
two isomorphisms are related to each other through an isomorphism
$\eta_t:V_t\cong \bar{V}_t$. 

According to Lemma \ref{ref-8.3-16} we may choose a sub line bundle
$\bar{\Fscr}\subset \bar{\Vscr}$ such that for all $t\in T$ we have
$\bar{\Fscr}_t=\eta_t(F_t)$.  Let $\Ascr'$ be the elementary transform
of $\Ascr$ associated to $(D,\bar{\Fscr},\bar{\Vscr})$.

Since $D$ is in $n\Pic(X)$ we have $\cl(\Ascr)=\cl(\Ascr')$ (using
\eqref{ref-7.1-15}). Let $\bar{\Qscr}=\bar{\Vscr}/\bar{\Fscr}$.

We have an exact sequence
\[
0\r \Bscr\otimes_X \Lscr\r \Ascr\otimes_X\Lscr\r \HHom_D(\bar{\Fscr},\bar{\Qscr})\otimes
\Lscr_D\r 0
\]
Checking at the points of $T$ we see that 
\[
H^0(X,\Bscr\otimes_X\Lscr)=H^0(X,\Lscr)
\]
Now we consider the other sequence
\[
0\r \Bscr\otimes_X \Lscr\r \Ascr'\otimes_X \Lscr\r \HHom_D(I_D/I^2_D\otimes
\bar{\Qscr},\bar{\Fscr})\otimes_D \Lscr_D\r 0
\]
where $I_D=\Oscr(-D)$. We deduce that there is an exact sequence
\[
0\r H^0(X,\Lscr)\r H^0(X,\Ascr'\otimes_X \Lscr)\r \Hom_D(I_D/I^2_D\otimes_D
\bar{\Qscr},\bar{\Fscr})\otimes_D \Lscr_D
\]
Now
\[
\Hom_D(I_D/I^2_D\otimes_D
\bar{\Qscr},\bar{\Fscr}\otimes_X \Lscr)\subset 
\Hom_D(I_D/I^2_D\otimes_D \bar{\Vscr},\bar{\Fscr}\otimes_X\Lscr_D)
\]
Choosing $\bar{\Fscr}$ negative enough we get
\[
 \Hom_D(I_D/I^2_D\otimes_D \bar{\Vscr},\bar{\Fscr}\otimes_X \Lscr_D)=0
\]
This finishes the proof.
\end{proof}
We have used the following version Bertini.
\begin{lemma} Let $X/k$ be a smooth projective variety of $\dim X\ge
  2$. Let $\Lscr$ be an ample line bundle on $X$. Let $T$ be a finite
  subset of $X$. Then for $n$ large enough the zeroes of a generic
  section of $\Lscr^{\otimes n}$ which is zero on $T$ will be smooth
  and connected.
\end{lemma}
\begin{proof}
Connectedness follows from \cite[Thm III.7.9]{H}. We prove smoothness
by suitably adapting \cite[Thm II.8.18]{H}.

If $\Mscr$ is a very ample line bundle on $X$ then for all $x\in X$ we
have that
\[
H^0(X,\Mscr)\r H^0(X,\Mscr\otimes \Oscr_X/m^2_x)
\]
is surjective.  This is in fact an equivalence. See \cite[Prop 7.3]{H}. It
follows easily that if $\Escr$ is a coherent sheaf generated by
global sections then
\[
H^0(X,\Escr\otimes_X\Mscr)\r H^0(X,\Escr\otimes\Mscr\otimes \Oscr_X/m^2_x)
\]
is surjective as well. 

It follows from Lemma \ref{ref-6.4-14} that for $n\gg 0$
we have that
\[
H^0(X,\Iscr\otimes_X\Lscr^n)\r H^0(X,\Iscr\otimes\Lscr^n\otimes \Oscr_X/m^2_x)
=H^0(X,(\Iscr\otimes\Lscr^n)/(\Iscr\otimes\Lscr^n)m^2_x)
\]
is surjective for all $x$. Imitating the proof of \cite[Thm
II.8.18]{H} we find that there is a Zariski dense subset $U\subset
H^0(X,\Iscr\otimes_X\Lscr^n)$ such that for $s\in U$ and $x\not\in T$
we have that the zeroes $Z$ of $s$ are smooth in $x$. Naturally $Z$
passes through $T$. It remains to prove that can make $Z$ smooth in
$T$ as well. Let $t\in T$. Then $Z$ is not smooth in $t$ if $s_t\in
H^0(X,\Lscr^n\otimes m^2_t)=H^0(X,(\Iscr\otimes\Lscr^n)m_t)$. Hence by
taking $s$ in the complement of a suitable set of linear spaces $Z$ is
smooth in $t$ as well. We are done.
\end{proof}

\section{General lifting}
\begin{proposition}
\label{ref-9.1-18}
Assume that we have a proper flat map of finite type $\phi:W\rightarrow C$ with $C/k$ of finite type as well. 
Let $x\in C$ and $X=\phi^{-1}(x)$. Let $\Ascr_0\in \Az_n(X)$ be such
that 
\[
H^2(X,\Ascr_0/\Oscr_X)=0
\]
 Then there
exists a diagram
\[
\begin{CD}
X=\phi^{\prime-1}(x)=\phi^{-1}(x)@>>> W'=C'\times_C W @>>> W\\
@VVV  @V\phi' VV @VV\phi V\\
x @>>> C' @>>\text{etale}> C
\end{CD}
\]
together with $\Ascr'\in \Az_n(W')$ such that 
$\Ascr'_{\phi^{\prime-1}(x)}=\Ascr_0$.
\end{proposition}
\begin{proof}
Let $m_x\subset \Oscr_C$ be the defining ideal for $x\in C$. 
The defining ideal of $X_0\overset{\text{def}}{=} X\subset W$ is 
\[
I\overset{\text{def}}{=}m_x\Oscr_W\overset{\text{flatness}}{\cong}
m_x\otimes_{\Oscr_C}\Oscr_W 
\]
Let $X_n$ be defined by $I^{n-1}\cong m^{n-1}_x\otimes_C\Oscr_W$. Then
$X_{n+1}$ inside $X_n$ is defined by 
\[
I^{n-1}/I^n\cong m^{n-1}_x/m^n_x \otimes_C\Oscr_W\cong (\Oscr_x)^{\oplus m_n}\otimes_C\Oscr_W\cong \Oscr_{X_0}^{\oplus m_n}
\]
Hence the obstruction of lifting an $\Ascr_n\in \Az(X_n)$ to an $\Ascr_{n+1}
\in \Az(X_{n+1})$ lies in 
\[
H^2(X_0, \Oscr_{X_0}^{\oplus m_n}\otimes_{X_0} \Ascr_0/\Oscr_{X_0})=0
\]
We may construct liftings 
\[
\cdots \rightarrow \Ascr_m\rightarrow \cdots \rightarrow\Ascr_2\rightarrow 
\Ascr_1\rightarrow \Ascr_0
\]
where $\Ascr_m$ lives in $\Az_n(X_m)$. Put

\vspace*{1mm}

\centerline{$%
\hat{\Ascr}=\invlim_n \Ascr
$}

\vspace*{1mm}

$\hat{\Ascr}$ is an Azumaya algebra over the formal completion
$\widehat{X}$ of $W$ at $X$. Recall that the formal completion of $X$
in $W$ is the ringed space$(X,\widehat{\Oscr}_{W,X})$.  This is not a
scheme but a so-called \emph{formal scheme} \cite{H}.

By Grothendieck's existence theorem \cite[5.1.4]{EGAIII} there is an
equivalence $\coh(\Spec \widehat{\Oscr}_{C,x}\times_C W)\cong
\coh(\widehat{X})$. Note that $\Spec \widehat{\Oscr}_{C,x}\times_C W$
is an actual scheme.

 From the Grothendieck existence theorem we  deduce easily $ \Az_n(\Spec
\widehat{\Oscr}_{C,x}\times_C W)=\Az_n(\widehat{X}) $. Let
$\tilde{\Ascr}\in \Az_n(\Spec \widehat{\Oscr}_{C,x}\times_C W)$
correspond to $\hat{\Ascr}$.

The functor
\[
\Sch/C:D\mapsto \Az_n(D\times_C W)
\]
is locally finitely presented in the sense that its restriction to
affine $k$-schemes commutes with filtered direct limits. It follows
from the Artin approximation theorem \cite[Cor (2.2)]{ArtinApprox}
that there exists $x\rightarrow C'\xrightarrow{\text{etale}} C$ and
$\Ascr'\in \Az_n(C'\times_C W)$ such that formally
\[
`` \Ascr'\equiv \tilde\Ascr \mod m_x''
\]
which means  
\[
\Ascr'\otimes_{\Oscr_{C'}}\Oscr_x\cong \tilde{\Ascr}\otimes_{\hat{\Oscr}_{C,x}}
\Oscr_x\cong \Ascr_0
\]
finishing the proof. 
\end{proof}
\section{Creating a family}
We prove the following result, 
\begin{proposition}
\label{ref-10.1-19} Let $X/k$ a smooth, proper, connected surface and let
$\alpha\in\Br(X)_n$. Then there exists
a smooth connected $k$-variety $W$ of dimension 3 and morphisms
\[
\begin{CD}
W@>g>> X\\
@VfVV\\
\AA^1
\end{CD}
\]
such that 
\begin{enumerate}
\item $f$ is smooth.
\item The generic fiber of $f$ is geometrically connected. 
\item $f$ is proper over a neighborhood of $0$.
\item $Y\overset{\text{def}}{=} W_0=f^{-1}(0)$ splits $\alpha$. 
\item $W_1=f^{-1}(1)\neq \emptyset$ and $g:W_1\r X$ is an open immersion.
\end{enumerate}
\end{proposition}

\begin{proof}[Proof of Proposition \ref{ref-10.1-19}]  We know that $\alpha_{k(X)}$ is represented
  by a division algebra.  Hence by lemma \ref{ref-5.8-11} we know
  there is an Azumaya algebra $\Ascr\in \Az_m(X)$ on $X$ representing
  $\alpha$ which is generically a division algebra.  We choose a
  line bundle $\Lscr$ and a section $s\in \Ascr\otimes_X \Lscr$ so that
  $Y_s$ is smooth and splits $\Ascr$.  Put
\[
\Bscr=\oplus_n \Lscr^{-n} T^n
\]
Then $Y_s$ is defined by the locally principal $\Bscr$ ideal generated
by $P(T)\Lscr^{-m}$ where $P(T)$ is the reduced characteristic polynomial
of $s$. 

Choose $w_1,\ldots,w_m$ distinct global sections of $\Lscr$ and put
\[
Q(T)=(T-w_1)\cdots (T-w_m)
\]
and
\[
R(t,T)=(1-t)P(T)+t Q(T)
\]
We view $R(t,T)$ as a section of 
\[
\Bscr[t]\otimes_X \Lscr^m
\]
We let $W'\subset L\times \AA^1$ be defined by the locally principal
$\Bscr[t]$ ideal generated by $R(t,T)\Lscr^{-m}$.  Thus we have
the following diagram
\[
\begin{CD}
W' @>\pr_1>> L @>>> X\\
@V \pr_2 VV\\
\AA^1
\end{CD}
\]
We denote the composition of the horizontal arrows by $g'$ and the 
vertical arrow by $f'$. 
We check
the following things.
\begin{enumerate}
\item
  $W'\r X\times \AA^1$ is finite flat. To see this we note that
  $R(t,T)$ is of the form $T^m+\text{lower terms in $T$}$. Thus
  $\Bscr[t]/(R(t,T)\Lscr^{-m})$ is locally free of rank $m$ over $X\times
\AA^1$.  For use below we record $W'=\underline{\Spec}\, \Cscr$ where
\[
\Cscr=\left(\bigoplus_{i=0}^m \Lscr^{-i}T^i)\right) [t]
\]
It follows in particular that $f':W'\r \AA^1$ is proper and flat. 
\item The fiber at $0$ of $f':W'\r \AA^1$  is equal to $Y_s$. This is
clear since this fiber is defined by $P(T)$.
\item The fiber at $1$ of $f':W'\r \AA^1$  is defined by $Q(T)$ and
can be written as $X=X_1\cup\cdots\cup X_m$ where $g'$ restricted
to each $X_i$ defines an isomorphism $X_i\cong X$. The singular
points of $(f')^{-1}(1)$ occur at the intersections $X_i\cap X_j$. This
happens at the points of $X$ where the sections $w_i$, $w_j$ are equal. 
\end{enumerate}
Now put
\[
W''=W'-X_2\cup\cdots \cup X_m
\]
Then the fibers of $f'|W''$ are smooth at $0,1$.  Since smoothness is
an open condition on flat maps \cite[Thm 17.5.1]{EGAIV} there is a
neighborhood $U$ of $0,1$ such that $(f')^{-1}(U)\r U$ is smooth. And
of course $(f')^{-1}(U-1)\r (U-1)$ is still proper. We let $W$ be the
smooth locus of $f'\mid W''$. Then $W$ contains $(f')^{-1}(U)$. We let
$f,g$ be the restrictions of $f',g'$ to $W$.

\medskip

The only property that remains to be proved is that the generic fiber
of $f$ is geometrically connected. It is clear that $f$ and $f'$ have
the same generic fiber hence we consider $f'$ which is proper.

Consider the Stein factorization
\cite{H} for proper morphisms.
\[
W'\xrightarrow{p} \Spec \Gamma(W',\Oscr_{W'})\xrightarrow{q} \AA^1
\]
Here $p$ has geometrically connected fibers \cite[Remarque 4.3.4]{EGAIII} and
$q$ is finite. We need to show that $q$ is an isomorphism. I.e. that
$\Gamma(W',\Oscr_{W'})=k[t]$. This follows from the fact that
\[
\Gamma(W',\Oscr_{W'})=\Gamma(X,\Cscr)=k[t]\qed
\]
\def\qed{}\end{proof}
\section{The unramified case}
\begin{theorem} Let $X$ be a projective smooth connected surface 
over algebraically closed field. 
 Assume that $\alpha\in \Br(X)\subset \Br(k(X))$
has period $n$. Then the index of $\alpha$ is $n$.
\end{theorem}
\begin{proof} We construct
\[
\begin{CD}
Y@>>>W@>g>> X\\
@VVV @VfVV\\
0@>>>\AA^1
\end{CD}
\]
as in Proposition \ref{ref-10.1-19}.

Choose $\eta\in H^2(X,\mu_n)$ whose image is $\alpha$. From the commutative
diagram
\[
\begin{CD}
0 @>>> \Pic(Y)/n\Pic(Y)@>>> H^2(Y,\mu_n) @>>> \Br(Y)_n @>>> 0\\
@. @AAA @AAA @AAA\\
0 @>>> \Pic(X)/n\Pic(X)@>>> H^2(X,\mu_n) @>>> \Br(X)_n @>>> 0
\end{CD}
\]
we see that we can choose $\Lscr\in \Pic(Y)$ such that the image
of $\Lscr$ in $H^2(Y,\mu_n)$ is $-\eta_Y$.  
Consider the following
Azumaya algebra on $Y$
\[
\Bscr_0=\HEnd(\Lscr_Y\oplus \Oscr_Y^{n-1})
\]
By Lemma \ref{ref-4.2-2} we find $\cl(\Bscr_0)=\eta_Y$. Using Lemma
\ref{ref-8.4-17} we may find an elementary transform
$\Ascr_0$ of $\Bscr_0$ such that $H^2(Y,\Ascr_0/\Oscr_Y)=0$ and such
that $\cl(\Ascr_0)=\cl(\Bscr_0)=\eta_Y$. According to Theorem
\ref{ref-9.1-18} there exists an etale neighborhood $C'\r \AA^1$ of $0$
such that $f':W'=C'\times_{\AA^1}W\r C'$ is proper and smooth
and an Azumaya algebra $\Ascr$ on $W'=C'\times_{\AA^1}W$ such that
$\Ascr_Y=\Ascr_0$.
Since $\cl(-)$ is compatible with base change we have $\cl(\Ascr)_Y=\cl(\Ascr_Y)
=\cl(\Ascr_0)=\eta_Y$.

We have $\cl(\Ascr)\in H^2(W',\mu_n)$. According to \cite[Prop 1.13]{Milne}  we
have
\[
(R^2f'_\ast \mu_n)_0=\dirlim H^2(C''\times_{C'} W',\mu_n)
\]
where $C''$ runs through all etale neighborhoods of $0\in C'$. By proper
base change \cite[IV-1]{arcata} we have
\[
(R^2f'_\ast \mu_n)_0=H^2(Y,\mu_n)
\]
Combining these two results we get
\[
H^2(Y,\mu_n)=\dirlim H^2(C''\times_{C'} W',\mu_n)
\]
Now $\cl(\Ascr)$ and $\eta_{W'}$ have the same image in $H^2(Y,\mu_n)$. It
follows there is some etale neighborhood $C''\r C'$ of $0\in C'$ such
that $\cl(\Ascr)_{W''}=\eta_{W''}$ where $W''=C''\times_{C'} W'$. We replace
$C''$ by $C'$, $W'$ by $W''$, $\Ascr$ by $\Ascr_{W''}$. We have now arrived at the situation where
$\cl(\Ascr)=\eta_{W'}$.  Going to Brauer groups we find
\[
[\Ascr]=[\cl(\Ascr)]=[\eta_{W'}]=[\eta]_{W'}=\alpha_{W'}
\]
Thus we have shown that the image of $\alpha\in \Br(k(X))$ in
$\Br(k(W'))$ has index $n$. By Theorem \ref{ref-3.1-0} it is now
sufficient to construct a discrete valuation ring in $k(W')$ with
residue field $k(X)$.

We proceed as follows. We may extend $C'\r \AA^1$ to finite morphism
of smooth (affine) curves $C\r \AA^1$. Then $\tilde{W}=C\times_{\AA^1} W$
is smooth over $C$ and hence regular. It contains $W'$ as an open
subset and hence $k(W')=k(\tilde{W})$. 

Since $k$ is algebraically closed there exists a point $c\in C$ lying
above $1\in \AA^1$. Then $c\times_{\AA^1} W$ is a divisor $D$ in
$\tilde{W}$ which is isomorphic to $0\times_{\AA^2} W$ which is an
open subset of $X$.  The discrete valuation ring in $k(\tilde{W})$
defined by $D$ has residue field $k(X)$ which is what we want.
\end{proof}
\section{The ramified case}
\begin{theorem} Let $k$ be an algebraically closed field. Assume that
$X/k$ is a smooth projective connected surface. Let $\alpha\in \Br(k(X))$
be of period $n$, prime to the characteristic of $k$. Then the 
index of $\alpha$ is $n$.
\end{theorem}
\begin{proof} It is standard that we may reduce to the case that
  $\period(\alpha)$ is prime $l$, different from the characteristic
  (we could already have done this earlier but for the unramified case
  does not simplify the proof). To avoid trivialities we assume
  $\alpha\neq 0$.

To start the proof we  construct morphisms
\[
\begin{CD}
W @>g>> X\\
@VfVV\\
\AA^1
\end{CD}
\]
of smooth connected varieties with $\dim W=3$ and with the following properties.
\begin{enumerate}
\item The morphism $f:W\r \AA^1$ is smooth. 
\item The generic geometric fiber $W_{\bar{\eta}}$ of $f:W\r \AA^1_k$ is
projective and connected. 
\item The fiber $W_1$ of $f$ in $1\in \AA^1_k$ is non-empty and the 
restriction of $g:W\r X$  to $W_1$ is birational.
\item The inverse image $r^\ast(\alpha)$ of $\alpha$ through the
composition $r:W_{\bar{\eta}}\r W_\eta\r W\r X$ is unramified on the
surface $W_{\bar{\eta}}$. 
\end{enumerate}
Assume we have constructed such morphisms. By the unramified we know
that $(\alpha)_{k(W_{\bar{\eta}})}$ has index $1$ or $l$. We can pull
this back to a finite extension of $k(\AA^1)$. Hence there exists a
finite extension $L/k(\AA^1)$ such that the extension of $\alpha$ to
$k(W_\eta\times_{k(\AA^1)}L)$ has index $1$ or $l$. Now $L$ is the
function field of a smooth curve $C$ such that there is a finite map $C\r \AA^1$. Hence
$k(W_\eta\times_{k(\AA^1)}L)$ is the function field of
$W'=C\times_{\AA^1} W$. So we have shown that the image of $\alpha\in
\Br(k(X))$ in $k(W')$ has index $1$ or $l$.  So what is left to do is to
construct a discrete valuation on $k(W')$ with residue field $k(W)$.
To do this we take a point $c\in C$ lying above $0\in \AA^1$. Then the
discrete valuation we want is the one associated to the divisor
$c\times_{\AA^1} W$.

\medskip

So what remains to be done is to construct the morphisms as indicated in
the diagram.

Let $D\subset X$ be the ramification divisor of $\alpha$. By
resolution of singularities we may assume that $D$ has normal
crossings. We claim that we can find smooth effective
divisors $E,E'$ such that $D+E$ has normal crossings and such that
\[
D+E\sim l(D+E')
\]
We choose  an ample line bundle $\Mscr$ on $X$. For $r$ sufficiently big
we have that $\Oscr_X((l-1)D)\otimes \Mscr^{rl}$ and $\Mscr^r$ are
very ample. Let $E'$ be the zeroes of a generic section of $\Mscr^r$. Then $E'$
is smooth by Bertini. Let $E$ be the zeroes of a generic section
of $\Oscr_X((l-1)D)\otimes \Mscr^{rl}$. Then $E$ is smooth and $D\cap E$
is smooth as well by Bertini (thus $D+E$ has normal crossings). In the last
application of Bertini we have used that a generic hyperplane section
will miss the finite number of singular points of $D\cap E$ \cite[II.8.18.1]{H}.
We now have isomorphisms
\begin{align*}
\Oscr_X(D+E')&\cong \Oscr_X(D)\otimes\Mscr^r\overset{\text{def}}{=}\Lscr \\
\Oscr_X(D+E)&\cong \Oscr_X(lD)\otimes\Mscr^{lr}=\Lscr^l
\end{align*}
Taking the images of $1$ under both isomorphisms we get sections $s_1
\in H^0(X,\Lscr)$, $s_0\in H^0(X,\Lscr^l)$ with zeroes respectively $D+E'$
and $D+E$. Note that $s_1^l$ and $s_0$ are both sections of $H^0(X,\Lscr^l)$.

As before we define 
\[
\Bscr=\bigoplus_{i\ge 0} \Lscr^{-i}T^i
\]
\[
L=\underline{\Spec}\, \Bscr
\]
We have
\[
L\times \AA^1=\underline{\Spec}\,\Bscr[t]
\]
and we put
\[
s_t=ts_1^l+(1-t)s_0
\]
Let $W'$ as the closed subscheme of $L\times \AA^1$ defined by 
$(T^l-s_t)\Lscr^{-l}\subset \Bscr[t]$. Completely analogous to the 
unramified case we get a diagram
\[
\begin{CD}
W' @>\pr_1>> L @>>> X\\
@V\pr_2 VV\\
\AA^1
\end{CD}
\]
We let the composition of the horizontal arrows be $g'$ and the vertical
arrow is $f'$.  As before we find that $f'$ is flat and proper. 

If we take the fiber of $W'\r X\times \AA^1$ at the generic point of
$X\times \AA^1$ we get
\[
\Spec k(X)(t)(\sqrt[l]{s_t})
\]
(to make sense of this one chooses an arbitrary trivialization 
$\Lscr_\eta\cong \Oscr_{X,\eta}$ at the generic point of $X$; in this
way $s_0,s_1$ become elements of $k(X)$).

Let $\tilde{D}=D\times \AA^1$. The valuations of $s_0, s^l_1$ for the
discrete valuation defined by an irreducible component of $\tilde{D}$ are
respectively $1$ and $l$. Hence $s_t$ has valuation $1$. Thus $s_t$ is
not an $l$'th power and hence $k(X)(t)(\sqrt[l]{s_t})$ is a field.
The same argument works if we replace $k(t)$ by a finite extension $L/k(t)$.
Hence we deduce that $f':W'\r \AA^1$ is generically geometrically irreducible. 

Let us now discuss the fibers of $0$ and $1$ of $f'$.
\begin{enumerate}
\item $(f')^{-1}(0)$ has equation $T^l=s_0$. Hence it is a degree $l$
cover of $X$, ramified in $D+E$.
\item $(f')^{-1}(1)$ has equation $T^l-s_1^l=0$ which factors in $l$
linear factors. We deduce that $(f')^{-1}(1)$  is of the form
\[
X_1\cup\cdots \cup X_l
\]
where $X_i\cong X$. The singular locus of $(f')^{-1}(1)$ is the locus
where $s_1$ is zero, i.e. $D+E'$. 
\end{enumerate}
We know that both $s_0$ and $s_1$ are zero on $D$. Hence $s_t$ is zero
on $\tilde{D}=D\times \AA^1$. The divisor defined by $s_t$ 
in $W'$ is of the form
\[
\tilde{D}+E_t
\]
There is an open neighborhood $U$ of $0$ where the fibers of $E_t$ and
$E_t\cap \tilde{D}$ have respectively dimension one and zero (this follows from
upper-semicontinuity of fiber dimension for proper morphisms; e.g.
combine \cite[I.8 Cor 3]{Mumford1} with the fact that proper morphisms
are closed \cite[Ch II]{H}). Since $X$ is regular it follows that $E_t$
and $E_t\cap D$ are defined locally by regular sequences and hence are Cohen-Macaulay.
In particular they have no embedded components. From  this one deduces that
$E_t/U$ and $D\cap E_t/U$ are flat (as over $k[t]$ torsion free modules are flat). 
By openness of smoothness for flat maps \cite[Thm 17.5.1]{EGAIV} we may shrink 
$U$ in such a way that $E_t$ and $E_t\cap\tilde{D}$ are smooth. In other words
the fibers of $E_t+\tilde{D}$ have normal crossings with smooth components.

\medskip

The equation of $(f')^{-1}(U)$ is (locally) $T^l-s_t$ and the zeroes
of $s_t\mid f^{-1}(U)$ are $(E_t+\tilde{D})\mid U$. Thus the
singularities of $(f')^{-1}(U)$ form a $1$-dimensional family of $A_{l-1}$
singularities. Using $\lceil \frac{l-1}{2}\rceil$ blowups of the closure of
the singular locus we may simultaneously resolve the fibers of
$f^{-1}(U)\r U$. We end up with a probably singular but integral
threefold $\tilde{f}:\tilde{W'}\r W'$ such that $\tilde{f}^{-1}(U)\r
U$ is smooth.

Now note that the ideals we have blown up define subvarieties of 
dimension 1. Since all components of $(f')^{-1}(1)$ are 
of dimension $2$ this means that $\tilde{f}^{-1}(1)$ will consist of
the strict transform of $(f')^{-1}(1)$ plus possibly some exceptional components. 

Now let $W''$ be obtained from $\tilde{W}'$ by removing from the fiber of
$1$ all components, except the strict transform of one
components of $W_1$. Finally let $W$ be the smooth locus of $W''$.
Then $W$ has all desired properties except that we still need to show
that $W_{\bar{\eta}}$ splits $\alpha$.

We have the diagram of maps
\[
W_{\bar{\eta}}=\tilde{W}'_{\bar\eta}\r W'_{\bar{\eta}}\r (X\times \AA^1)_{\bar{\eta}}
\]
Here $W'_{\bar{\eta}}/(X\times \AA^1)_{\bar{\eta}}$ is a cyclic
covering which is ramified in a normal crossing divisor with smooth
components which contains the ramification locus of $\alpha$ and
$\tilde{W}'_{\bar\eta}$ is the canonical desingularization of
$W'_{\bar{\eta}}$ obtained by repeatedly blowing up the singular
locus.  It now suffices to invoke Proposition \ref{ref-5.7-10}.
\end{proof}

\def\cprime{$'$} \def\cprime{$'$} \def\cprime{$'$}
\ifx\undefined\bysame
\newcommand{\bysame}{\leavevmode\hbox to3em{\hrulefill}\,}
\fi

\end{document}